\documentclass{amsart}
\usepackage{latexsym,amsxtra,amscd,ifthen}
\usepackage{amsfonts}
\usepackage{verbatim}
\usepackage{amsmath}
\usepackage{amsthm}
\usepackage{amssymb}
\numberwithin{equation}{section}

\usepackage{cases}
\usepackage{amsmath}
\usepackage{amssymb}
\usepackage[all]{xy}
\usepackage{mathrsfs}
\usepackage{amsfonts}
\usepackage{amscd}
\usepackage{multirow,array}

\usepackage{amssymb}
\usepackage{graphicx, picinpar,epsf}
\usepackage[colorlinks, linkcolor=red,anchorcolor=blue,citecolor=green]{hyperref}

\begin{document}
\input xy
\xyoption{all}

\renewcommand{\mod}{\operatorname{mod}\nolimits}
\newcommand{\proj}{\operatorname{proj}\nolimits}
\newcommand{\rad}{\operatorname{rad}\nolimits}
\newcommand{\soc}{\operatorname{soc}\nolimits}
\newcommand{\ind}{\operatorname{ind}\nolimits}
\newcommand{\Top}{\operatorname{top}\nolimits}
\newcommand{\ann}{\operatorname{Ann}\nolimits}
\newcommand{\id}{\operatorname{id}\nolimits}
\newcommand{\Mod}{\operatorname{Mod}\nolimits}
\newcommand{\End}{\operatorname{End}\nolimits}
\newcommand{\Ob}{\operatorname{Ob}\nolimits}
\newcommand{\Ht}{\operatorname{Ht}\nolimits}
\newcommand{\cone}{\operatorname{cone}\nolimits}
\newcommand{\rep}{\operatorname{rep}\nolimits}
\newcommand{\Ext}{\operatorname{Ext}\nolimits}
\newcommand{\Hom}{\operatorname{Hom}\nolimits}
\renewcommand{\Im}{\operatorname{Im}\nolimits}
\newcommand{\Ker}{\operatorname{Ker}\nolimits}
\newcommand{\Coker}{\operatorname{cok}\nolimits}
\renewcommand{\dim}{\operatorname{dim}\nolimits}
\newcommand{\Iso}{\operatorname{Iso}\nolimits}
\newcommand{\Ab}{{\operatorname{Ab}\nolimits}}
\newcommand{\Coim}{{\operatorname{Coim}\nolimits}}
\newcommand{\pd}{\operatorname{pd}\nolimits}
\newcommand{\sdim}{\operatorname{sdim}\nolimits}
\newcommand{\add}{\operatorname{add}\nolimits}
\newcommand{\cc}{{\mathcal C}}
\newcommand{\co}{{\mathcal O}}
\newcommand{\crc}{{\mathscr C}}
\newcommand{\crd}{{\mathscr D}}
\newcommand{\crm}{{\mathscr M}}
\newcommand{\coh}{\operatorname{coh}\nolimits}
\newcommand{\vect}{\operatorname{vect}\nolimits}
\newcommand{\cub}{\operatorname{cub}\nolimits}
\newcommand{\can}{\operatorname{can}\nolimits}
\newcommand{\rk}{\operatorname{rk}\nolimits}
\newcommand{\st}{[1]}
\newcommand{\X}{\mathbb{X}}
\newcommand{\N}{\mathbb{N}}
\newcommand{\ul}{\underline}
\newcommand{\vx}{\vec{x}}
\newcommand{\vy}{\vec{y}}
\newcommand{\vz}{\vec{z}}
\newcommand{\vc}{\vec{c}}
\newcommand{\vw}{\vec{w}}
\def\proof{\vspace{1mm}\noindent{\it Proof}\quad}

\newtheorem{theorem}{Theorem}[section]
\newtheorem{acknowledgement}[theorem]{Acknowledgement}
\newtheorem{algorithm}[theorem]{Algorithm}
\newtheorem{axiom}[theorem]{Axiom}
\newtheorem{case}[theorem]{Case}
\newtheorem{claim}[theorem]{Claim}
\newtheorem{conclusion}[theorem]{Conclusion}
\newtheorem{condition}[theorem]{Condition}
\newtheorem{conjecture}[theorem]{Conjecture}
\newtheorem{corollary}[theorem]{Corollary}
\newtheorem{criterion}[theorem]{Criterion}
\newtheorem{definition}[theorem]{Definition}
\newtheorem{example}[theorem]{Example}
\newtheorem{exercise}[theorem]{Exercise}
\newtheorem{lemma}[theorem]{Lemma}
\newtheorem{notation}[theorem]{Notation}
\newtheorem{proposition}[theorem]{Proposition}
\newtheorem{remark}[theorem]{Remark}
\newtheorem{solution}[theorem]{Solution}
\newtheorem{summary}[theorem]{Summary}
\newtheorem*{thma}{Theorem}
\newtheorem*{Riemann-Roch Formula}{Riemann-Roch Formula}
\newtheorem{question}[theorem]{Question}
\numberwithin{equation}{section}
\title{On tubular tilting objects in the stable category of vector bundles}      

\author[Chen]{Jianmin Chen$^\dag$}
\address{School of Mathematical Sciences, Xiamen University, Xiamen 361005, P.R.China}
\thanks{$^\dag$ Corresponding author}
\email{chenjianmin@xmu.edu.cn}

\author[Lin]{Yanan Lin}
\address{School of Mathematical Sciences, Xiamen University, Xiamen
361005, P.R.China} \email{ynlin@xmu.edu.cn}

\author[Ruan]{Shiquan Ruan}
\address{School of Mathematical Sciences, Xiamen University, Xiamen
361005, P.R.China} \email{sqruan@xmu.edu.cn}

\begin{abstract}
The present paper focuses on the study of the stable category of vector bundles for the weighted projective lines of weight triple. We find some important triangles
in this category and use them to construct tilting objects with tubular endomorphism algebras for the case of genus one via cluster tilting theory.    
\end{abstract}



\maketitle

\section{Introduction}

Weighted projective lines, introduced by Geigle and Lenzing, establish a link between many mathematical subjects such as representation theory of
algebras \cite{GL}, automorphic
forms\cite{L}, and singularities\cite{KLM2}. Let $\mathbb{X}$ be a weighted projective line over an
algebraically closed field $k$. By \cite{GL}, the category of coherent sheaves on $\X$ is derived equivalent to the category of finite-dimensional modules over some canonical algebra, and the category of vector bundles on $\X$, as an additive category, is equivalent to the category of graded Cohen-Macaulay modules over its corresponding graded ring.

The present paper focuses on the study of a weighted projective line $\X$ of weight triple $(p_{1},p_{2},p_{3})$.
Kussin, Lenzing and Meltzer
\cite{KLM2} proved that the category  $\vect\X$ of vector bundles
on $\mathbb{X}$, under the distinguished exact structure, is a Frobenius category with the system $\mathcal{L}$ of all
line bundles as the system of all indecomposable
projective-injectives. Moreover, the attached stable category
$\underline{\vect}\X=\vect\X/[\mathcal{L}]$ is triangulated.
In particular, this triangulated category is closely related to the categories of finitely generated modules over Nakayama algebras,
the stable categories of graded maximal
Cohen-Macaulay modules and the singularity
categories of some graded rings.
So it is important and interesting to study the structure of this triangulated category, especially the triangles and tilting objects.

 The excellent reference \cite{KLM2} is devoted to understanding the structure of $\underline{\vect}\X$. In particular, Kussin-Lenzing-Meltzer found a triangle consisting of rank-two bundles
(see Section 2 for the notations)
\[ E_{L}\langle\vec{x}\rangle\to E_{L}\langle\vec{x}+\vec{x}_{i}\rangle\to
 E_{L}\langle\vec{x}-l_{i}\vec{x}_{i}\rangle((l_{i}+1)\vec{x}_{i})\to E_L\langle\vec{x}\rangle[1].
\]
 This triangle plays an important role in their construction of the tilting object \[T_{\cub}=\bigoplus\limits_{0\leq \vec{x}\leq 2\vec{\omega}+\vec{c}}E\langle\vec{x}\rangle,\] whose endomorphism algebra is $k\vec{A}_{p_{1}-1}\otimes k\vec{A}_{p_{2}-1}\otimes k\vec{A}_{p_{3}-1}$.
But there is still much unknown for this triangulated category.
The aim of this paper is to find more triangles and tilting objects.

Notice by \cite{O} that the stable category $\underline{\vect}\X$ is equivalent to the derived category of coherent sheaves on $\X$ if  $\mathbb{X}$ is of genus one, that is, $\X$ is of weight type $(2,4,4), (2,3,6)$ or $(3,3,3)$. This implies that there exists a tilting object in $\underline{\vect}\X$ such that its endomorphism algebra is a canonical algebra of tubular type. We aim to construct such a tilting object (called \emph{tubular tilting object}).

The main idea is to use cluster tilting theory. As we know that an advantage of cluster tilting theory over classical tilting theory
is that there is an important tool, named cluster mutation, in cluster categories. Thus the
usual procedure of going from a tilting object to another one by
exchanging just one indecomposable direct summand gets more regular. We will construct the desired tilting object in $\underline{\mbox{vect}}\mathbb{X}$ from a cluster tilting object in its cluster category.

We would like to emphasize that although there exists a tilting object $T_{\cub}$ in the stable category $\underline{\mbox{vect}}\mathbb{X}$,
it is not clear whether $T_{\cub}$ is a cluster tilting object in its cluster category. By studying the properties of tilting objects in the
stable category, we show that $T_{\cub}$ is a cluster tilting object for weight types $(2,4,4)$ and $(2,3,6)$, but not for weight type $(3,3,3)$. We construct a cluster tilting object for weight type $(3,3,3)$ as the original object for cluster mutation.

This paper is organized as follows: In Section 2, we recall some notions for later use. In Section 3, we present some triangles
in the stable category of vector bundles on a weighted projective line of weight triple.
The final section is devoted to studying tilting objects in the stable category. By using the cluster tilting mutations and triangles given in  Section 3,
we construct a tubular tilting object in the stable category of vector bundles for each weighted projective line of genus one and weight triple case by case.
In view of the work on \cite{CLR}, which constructed a tubular tilting object in $\underline{\mbox{vect}}\mathbb{X}$ of weight type $(2,2,2,2)$, we actually realize the construction of a tubular tilting object for all the weighted projective lines of genus one.

\section{Preliminaries}

In this section, we list some basic definitions concerning the weighted projective lines and cluster tilting theory. The properties we will use can be found in \cite{CLLR, GL, KLM2}.

\subsection{The category of coherent sheaves $\coh \X$}
Let $k$ be an algebraically closed field. A \emph{weighted
projective line} $\X$ over $k$ is specified by giving a collection
$\lambda=(\lambda_{1},\lambda_{2},\cdots, \lambda_{n}) $ of distinct
points in the projective line $\mathbb{P}^{1}(k)$, and a
\emph{weight sequence} $p=(p_{1},p_{2},\cdots, p_{n})$ of positive
integers. The associated rank one abelian group $\mathbb{L}$ has
generators $\vec{x}_{1}, \vec{x}_{2}, \cdots, \vec{x}_{n}$ with the
relations $p_{1}\vx_1=p_{2}\vx_2=\cdots=p_{t}\vx_n=:\vec{c}.$
Each element $\vec{x}\in \mathbb{L}$ can be
uniquely written in normal form
$$\vec{x}=\sum\limits_{i=1}^{n}l_{i}\vec{x}_{i}+l\vec{c},
\ \ \mbox{where} \ \ 0\leq  l_{i}< p_{i} \ \ \mbox{and} \ \ l\in
\mathbb{Z}.$$
The associated commutative algebra
\[S=k[X_{1},X_{2},\cdots,
X_{n}]/(f_{3},\cdots,f_{n}):= k[x_{1},x_{2}, \cdots, x_{n}],\] where
$f_{i}=X_{i}^{p_{i}}-X_{2}^{p_{2}}+\lambda_{i}X_{1}^{p_{1}},
i=3,\cdots, n,$   is $\mathbb{L}$-graded by setting
$\mbox{deg}(x_{i})=\vx_i.$

The category of coherent sheaves on $\X$ is the
quotient of the category of finitely generated $\mathbb{L}$-graded
$S$-modules over the Serre subcategory of finite length modules
\[\coh \X:=\mod^{\mathbb{L}}(S)/\mbox{mod}_{0}^{\mathbb{L}}(S). \]
The free module $S$ gives the structure sheaf $\co$, and each line
bundle is given by the grading shift $\co(\vec{x})$ for a uniquely
determined element $\vec{x}\in \mathbb{L}$. Moreover, there is a
natural isomorphism
\[\Hom(\co(\vec{x}), \co(\vec{y}))=S_{\vec{y}-\vec{x}}.
\]

The structure of $\coh \X$ was described in \cite{GL}. Especially, $\coh \X$ is a hereditary abelian
category with Serre duality of the form
\[D\Ext^1(X, Y)=\Hom(Y, X(\vec{\omega})),\] where $\vec{\omega}=(n-2)\vec{c}-\sum\limits_{i=1}^{n}\vec{x}_{i}$ is
called the \emph{dualizing element} of $\mathbb{L}$.

The Grothendieck group $K_0(\X)$ of $\coh\X$ was computed by \cite{GL}, and the definitions of homomorphism $\delta,$ \emph{determinant} map $\det$, \emph{rank} function
$\rk$, \emph{degree} function $\deg$, \emph{slope} function $\mu$ can also be found in \cite{GL}.

\subsection{Stable category of vector bundle $\underline{\vect}\X$}
We shall always assume that
$\mathbb{X}$ is a weighted projective line of weight triple in the rest of the paper. Denote by $\vect\X$ the full subcategory of $\coh\X$ formed by all
vector bundles. A sequence
\[0\rightarrow X^{\prime}\rightarrow X\rightarrow
X^{\prime\prime}\rightarrow 0\] in $\mbox{vect}\mathbb{X}$ is called
\emph{distinguished exact} if for each line bundle $L$ the induced
sequence
\[
0\rightarrow \Hom(L,X^{\prime})\rightarrow \Hom(L, X)\rightarrow
\Hom(L, X^{\prime\prime})\rightarrow 0\] is exact. The distinguished exact
sequences define a Frobenius exact structure on
$\vect\X$ such that the system of all line bundles is the system of
all indecomposable projective-injectives. By a general result of
\cite{H}, the related stable category
$\underline{\vect}\X=\vect\X/[\mathcal{L}]
$ is a triangulated category.
For simplicity of notations, in the rest of the paper we denote
the stable category $\underline{\vect}\X$ by $\crd$ and denote the
homomorphism space between $X$ and $Y$ in $\crd$ by $\crd(X, Y)$.

By \cite{KLM2},
$\mathscr{D}$ is a
$\Hom$-finite, homologically finite, Krull-Schmidt category with Serre duality:
$\mathscr{D}(X, Y[1])=D\mathscr{D}(Y, X(\vec{\omega}))
$ for any two objects $X$ and $Y$ in $\crd$.
For a line bundle $L$ and $0\leq \vx \leq 2\vec{\omega}+\vc=\sum\limits_{i=1}^{3}(p_{i}-2)\vec{x}_{i}$, the middle term of the non-split exact sequence
\[0\rightarrow L(\vec{\omega})\rightarrow \bullet \rightarrow
L(\vx)\rightarrow 0\] in $\mbox{vect}\mathbb{X}$ is unique up to isomorphism. We call it the \emph{extension bundle} determined by $(L, \vx)$,
and denote it by $E_{L}\langle \vx \rangle$. We simply denote by $E_{L}:=E_{L}\langle 0 \rangle$ and call it the \emph{Auslander bundle} associate with $L$,
denote by $E\langle \vx \rangle:=E_{\co}\langle \vx \rangle$. In particular, $E:=E_{\co}\langle 0 \rangle.$

Details about the structure of $\crd$, the injective hull, the projective cover, the suspension for extension bundles, and homomorphism spaces between extension bundles are given in \cite{KLM2}. Based on the work of \cite{KLM2}, Lenzing and the third-named author \cite{LR} further studied extension bundles in $\underline{\vect}\X$, and obtained the following basic property. For the convenience of the readers, we sketch the proof.

\begin{lemma}\label{lemma1} Assume $\vec{x},\vec{y}, \vec{z}\in \mathbb{L}, \vec{x}=\sum\limits_{i=1}^{3}l_{i}\vec{x}_{i}$ with $0\leq l_{i}\leq p_{i}-2$ for $i=1,2,3$. Then $E\langle\vec{x}\rangle=E\langle\vec{y}\rangle(\vec{z})$ if and only if one of the following
conditions holds
\begin{enumerate}
\item[-] $\vec{y}=\vec{x}$ and $\vec{z}=0$;
\item[-] $\vec{y}=l_{j}\vec{x}_{j}+\sum\limits_{i\neq j}(p_{i}-2-l_{i})\vec{x}_{i}$ and $\vec{z}=\sum\limits_{i\neq j}(l_{i}+1)\vec{x}_{i}-\vec{c}$ for some $1\leq j\leq3$.
\end{enumerate}
\end{lemma}

\begin{proof}
Assume that $\vy=\sum\limits_{i=1}^3 k_i\vx_i$ and
$\vz=\sum\limits_{i=1}^3\lambda_i\vx_i+\lambda \vc$ are both in
normal forms. Then $E\langle \vx\rangle=E\langle \vy\rangle (\vz)$
if and only if they have the same class in the Grothendieck group
since they are both exceptional in $\coh\X$ (\cite{KLM2}, Theorem 4.2), hence we get
$$[\co(\vw)]+[\co(\vx)]=[\co(\vw+\vz)]+[\co(\vy+\vz)],$$ that is,
\begin{equation}\label{classes expression}
\begin{array}{ll}
&[\co(\sum\limits_{i=1}^{3}(p_i-1)\vx_i-2\vc)]+[\co(\sum\limits_{i=1}^{3}l_i\vx_i)]\\
=& [\co(\sum\limits_{i=1}^3(\lambda_i-1)\vx_i+(\lambda
+1)\vc)]+[\co(\sum\limits_{i=1}^3(k_i+\lambda_i)\vx_i+\lambda \vc)].
\end{array}
\end{equation}
Representing each summand of (\ref{classes
expression}) by the basis $\{[\co(\vx)] | 0\leq \vx \leq \vc \}$ and comparing the determinants in both sides, we obtain the result.
\end{proof}

Recall that an object $T$  in $\crd$ is called \emph{tilting} if
\begin{enumerate}
\item[-]$T$ is extension-free, i.e., $\mathscr{D}(T, T[n])=0$ for each non-zero integer $n$.
\item[-]$T$ generates the triangulated category $\mathscr{D}$, i.e.,
the smallest thick triangulated subcategory $\langle T \rangle$ containing
$T$ is $\crd$.
\end{enumerate}

It is not easy to construct tilting objects in the stable category since all the line bundles are killed in $\crd$ and the minimal rank of the objects in $\crd$ is two. By investigating homomorphism spaces and important triangles for rank-two bundles in $\crd$, Kussin-Lenzing-Meltzer finally obtained a tilting object consisting of rank-two bundles as follows:

\begin{lemma}[Tilting cuboid, \cite{KLM2}]\label{lemma7} Let $L$ be a line bundle. Then
\[T_{\cub}(L)=\bigoplus\limits_{0\leq \vx\leq 2\vec{\omega}+\vec{c}}E_{L}\langle\vx\rangle\]
is a tilting object in $\mathscr{D}$ with endomorphism ring
\[\underline{\End}(T_{\cub}(L))\cong k\vec{A}_{p_{1}-1}\otimes k\vec{A}_{p_{2}-1}\otimes k\vec{A}_{p_{3}-1}.\]
\end{lemma}

We denote by $T_{\cub}=T_{\cub}(\co)$ in the rest of paper.

\subsection{Relationship to cluster tilting theory}\label{sectiion cluster category}

We assume $\mathbb{X}$ is a weighted projective line $\mathbb{X}$ of genus one and weight triple in this subsection, that is, $\mathbb{X}$ is of weight type
$(2,3,6), (2,4,4)$ or $(3,3,3)$. We have
Riemann-Roch Formula \cite{LM}, the tubular factorization property
and a bijective,
monotonous map $\alpha : \mathbb{Q}\to \mathbb{Q}$ with $
\alpha(q)>q$ for all $q\in \mathbb{Q} $ such that
$\mu(X[1])=\alpha(\mu(X))$ for each indecomposable vector bundle $X$.
Precisely, $\alpha^{-1}(0)=-\frac{3}{2}$ for weight type $(3,3,3)$, $\alpha^{-1}(0)=-2$ for weight type $(2,4,4)$ and $\alpha^{-1}(0)=-3$ for weight type $(2,3,6)$ (\cite{KLM2}, Theorem A.3).
For any $a\in \mathbb{Q}$, the interval
category $\crd(a, \alpha(a)]$ which is the full subcategory of $\crd$
obtained as the additive closure of all the indecomposable objects
with slopes in the interval $(a, \alpha(a)]$, is an abelian category
and equivalent to $\coh\X$.

The following lemma is useful.

\begin{lemma}[\cite{CLLR}]\label{lemma6}Let $T=\oplus T_i$ be an object in $\crd$ with indecomposable direct summands $T_i\in\crd(a, \alpha(a)]$ for some $a\in
\mathbb{Q}$. Then $T$ is extension-free in $\mathscr{D}$ if and only
if $\mathscr{D}(T, T[1])=0$.
\end{lemma}

According to \cite{O}, the bounded derived category $D^b(\coh\X)$ of coherent
sheaves  is
triangle equivalent to the stable category
$\crd$. Thus parallel to \cite{BKL}, we define the cluster
category $\crc$ to be the orbit category of the stable category
$\crd$ under the action of the unique auto-equivalence $G=\tau^{-1}\st$. The cluster category $\crc$ has the same
objects as $\crd$, and for any objects $X, Y$, the homomorphism spaces are
given by \[\crc(X, Y)=\bigoplus_{n\in \mathbb{Z}}\mathscr{D}(X,
G^{n}Y)\] with the obvious composition. This orbit category is
triangulated and Calabi-Yau of CY-dimension 2, and the canonical
functor $\pi: \crd\to\crc$ is a triangulated functor\cite{K}.

From \cite{KR} that, an object $T$ in
$\mathscr{C}$ is called a \emph{cluster tilting} object if
\begin{enumerate}
\item[-]$\mathscr{C}(T, T[1])=0$.
\item[-]$\mathscr{C}(T,
X[1])=0$ implies $X \in \mbox{add}(T)$.
\end{enumerate}

Let $T=\oplus T_i$ be an object in $\crd$
with each indecomposable direct summand $T_i\in\crd(a, \alpha(a)]$
for some $a\in \mathbb{Q}$. The following result is from \cite{BKL} (see also \cite{CLLR}).

\begin{lemma}\label{lemma9} The object $T$ is a tilting object in
$\mathscr{D}$ if and only if $T$ is a cluster tilting object in
$\mathscr{C}$.
\end{lemma}

The following lemma is useful.

\begin{lemma}[\cite{BMRRT,IY}]\label{cluster mutation}Let $\cc$ be a $\Hom$-finite 2-CY triangulated category
with a cluster tilting object $T$. Let $T_i$ be indecomposable and
$T=T_0\oplus T_i$. Then there exists a unique indecomposable $T_i^*$
non-isomorphic to $T_i$ such that $\crm_{T_i}(T)=T_0\oplus T_i^*$ is cluster
tilting. Moreover $T_i$ and $T^*_i$ are linked by the existence of
exchange triangles
\[T_i\xrightarrow{u}B\xrightarrow{v}T^*_i\xrightarrow{w}T_i\st\quad\text{and}\quad
T^*_i\xrightarrow{u'}B'\xrightarrow{v'}T_i\xrightarrow{w'}T^*_i\st,
\]
where $u$ and $u'$ are minimal left $\add T_0$-approximations and $v$ and $v'$ are minimal
right $\add T_0$-approximations.
\end{lemma}

This recursive process of mutations for cluster tilting objects is
closely related to the notion of mutations of quivers. Recall that in \cite{FZ}
the mutation
of a finite quiver $Q$ without loops and oriented cycles of length 2 (2-cycles for short) at a vertex $i$ is a quiver denoted by $\crm_i(Q)$ and constructed from $Q$ using the following rule:
\begin{itemize}
\item[(M1)] for any couple of arrows $j \to i \to k$, add an arrow $j \to k$;
\item[(M2)] reverse the arrows incident with $i$;
\item[(M3)] remove a maximal collection of 2-cycles.
\end{itemize}

\begin{lemma}[\cite{BIRS}]\label{theorem BIRS}
Let $\cc$ be a $\Hom$-finite 2-CY triangulated category with a
cluster tilting object $T$. Let $T_i$ be an indecomposable direct
summand of $T$, and denote by $T'$ the cluster tilting object
$\crm_{T_i}(T)$. Denote by $Q_T$ (resp. $Q_{T'}$) the quiver of the
endomorphism algebra $\End_\cc(T)$ (resp. $\End_\cc(T')$). Assume
that there are no loops and no 2-cycles at the vertex $i$ of $Q_T$
(resp. $Q_{T'}$) corresponding to the indecomposable $T_i$ (resp.
$T^*_i$ ). Then
$Q_{T'}=\crm_i(Q_T).$
\end{lemma}

\section{Important triangles in $\underline{\vect}\mathbb{X}$}

In this section, we will present some triangles, which are crucial for constructing tilting objects
in the stable category $\mathscr{D}:=\underline{\vect}\mathbb{X}$ of weight type $(p_{1},p_{2},p_{3})$. Denote by $\overline{x}_{i}=\vec{x}_{i}+\vec{\omega}$ in $\mathbb{L}$ for $i=1,2,3$.

According to (\cite{KLM2}, Corollary 4.14), $\crd(E, E(\vec{x}))\neq 0$ if and only if $\vec{x}=0$ or $\overline{x}_{i}$ for $i=1,2,3$. Moreover, $\crd(E, E(\overline{x}_{i}))\cong k$ for each $i$. Hence each non-zero morphism $E\to E(\overline{x}_{i})$ fits into a triangle in $\crd$. In fact, let $\{i,j,k\}=\{1,2,3\}$, we have

\begin{proposition}\label{proposition1}
For any $1\leq i\leq 3$, there exists a triangle in $\crd$:
\[E\to E(\overline{x}_{i})\to F_{i}\to E[1].
\]
Here, $F_{i}$ depends on the weight type. In more details, $F_{i}$ is given as follows:
\begin{itemize}
\item[(1)] if $p_{j}=p_{k}=2$, then $F_{i}=0$;
\item[(2)] if $p_{j}=2$ and $p_{k}>2$, then $F_{i}=E\langle (p_{k}-3)\vec{x}_{k}\rangle (\vec{x}_{k})$;
\item[(3)] if $p_{j}, p_{k}>2$, then $F_{i}$ is determined by the non-split exact sequence (for any $t\neq i$):
\[\xi_{i,t}: 0 \to
E\langle (p_{t}-3)\vec{x}_{t}\rangle (\vec{x}_{t})\to F_{i}
\to \co(\overline{x}_{i}+\overline{x}_{t})\to 0;
\]
\end{itemize}

moreover, $\rk F_{i}=3$ and
\[PF_{i}=(\bigoplus\limits_{t\neq i}(\co(\overline{x}_{i}-\vec{x}_{t})\oplus \co(\overline{x}_{t})))\oplus\co(\vec{\omega}+\overline{x}_i)\oplus \co,
\]
and
\[IF_{i}=\bigoplus_{t=1}^{3}(\co((p_{t}-1)\vec{x}_{t})\oplus \co(\overline{x}_{i}+\overline{x}_t)).
\]
\end{proposition}

\begin{proof}
From (\cite{KLM2}, Corollary 4.14), we have $\crd(E, E(\overline{x}_{i}))\neq 0$ for $1\leq i\leq 3$. Then by \cite{GL}, any nonzero morphism $f: E\to E(\overline{x}_{i})$ is injective in
$\coh\mathbb{X}$. Thus we get a short exact sequence
\[0 \to
E \to E(\overline{x}_{i})
\to \mathcal{S} \to 0
\]
where $\mathcal{S}$ is a sheaf of finite length.

Notice that by Lemma 4.10 in \cite{KLM2}, $\det \mathcal{S}= \det E(\overline{x}_{i})-\det E=2 \overline{x}_{i}= (p_{j}-2)\vec{x}_{j}+ (p_{k}-2)\vec{x}_{k}$. We have the following three cases to
consider:

(1)  If $p_{j}=p_{k}=2$, then $\det \mathcal{S}=0$. It follows that $\mathcal{S}=0$, and then $E\cong E(\overline{x}_{i})$. Hence $F_{i}=0$.

(2)  If $p_{j}=2$ and $p_{k}>2$, then it is easy to verify that $\mathcal{S}= S_{k,p_{k}-2}^{(p_{k}-2)}$, where $S_{k,p_{k}-2}$ is the unique simple sheaf concentrated at the exceptional point corresponding to $x_{k}$ satisfying $\Hom(\co((p_{k}-2)\vec{x}_{k}), S_{k,p_{k}-2})\neq 0$, and $S_{k,p_{k}-2}^{(p_{k}-2)}$
    is the unique torsion sheaf with top $S_{k,p_{k}-2}$ and of length $p_{k}-2$.

Now we make the following pushout commutative diagram:
\[
\xymatrix{
0 \ar[r] & E  \ar[d]\ar[r] & E(\overline{x}_{i})\ar[d] \ar[r]\ar@{}[dr] &  \mathcal{S} \ar@{=}[d]\ar[r] & 0  \\
0 \ar[r] & IE \ar[r] & F\ar[r] & \mathcal{S} \ar[r]& 0 , \\
 }
\]
where $IE=\co\oplus(\bigoplus\limits_{i=1}^{3}\co(\overline{x}_{i}))$ is the injective hull of $E$.

Notice that for any $t\neq k$, $\Ext^{1}(\mathcal{S}, \co(\overline{x}_{t}))=0$. We obtain that $\co(\overline{x}_{t})$ is a direct summand of $F$ for $t\neq k$.
By canceling out the common line bundle summands of $IE$ and $F$, we get an exact sequence
\[0 \to
\co\oplus \co(\overline{x}_{k}) \to F_{i}
\to \mathcal{S} \to 0
\]
Observe that $F_{i}$ is indecomposable of rank two, which is an extension bundle. According to (\cite{KLM2}, Theorem 4.2), $F_{i}$ is determined by its class in $K_{0}(\mathbb{X})$.
Moreover, $[F_{i}]=[\co]+[\co(\overline{x}_{k})]+[S_{k,p_{k}-2}^{(p_{k}-2)}]=[\co(\overline{x}_{k})]+[\co((p_{k}-2)\vec{x}_{k})]=[E\langle (p_{k}-3)\vec{x}_{k}\rangle (\vec{x}_{k})].$ Thus we get $F_{i}=E\langle (p_{k}-3)\vec{x}_{k}\rangle (\vec{x}_{k}).$

(3)  If $p_{j}, p_{k}>2$, then $\mathcal{S}=S_{j,p_{j}-2}^{(p_{j}-2)}\oplus S_{k,p_{k}-2}^{(p_{k}-2)}.$ In the following pushout diagram,
\[
\xymatrix{
0 \ar[r] & E  \ar[d]\ar[r] & E(\overline{x}_{i})\ar[d] \ar[r]\ar@{}[dr] &  \mathcal{S} \ar@{=}[d]\ar[r] & 0  \\
0 \ar[r] & IE \ar[d]\ar[r] & F\ar[d]\ar[r] & \mathcal{S} \ar[r]& 0,   \\
 & E[1] \ar@{=}[r] & E[1] & &
 }
\]
we find that only one of the direct summands, $\co(\overline{x}_{i})$, of $IE$ is a direct summand of $F$ since $\Ext^{1}(\mathcal{S}, \co(\overline{x}_{i}))=0.$ Hence we get distinguished exact sequences
\[\zeta_{i}: 0 \to
E\to E(\overline{x}_{i})\oplus IE
\to F\to 0,
\]
and
\[\gamma_{i}: 0 \to
E(\overline{x}_{i})\to \co(\overline{x}_{i})\oplus F_{i}
\to E[1]\to 0,
\]
here, $E[1]$ is viewed as an object in $\vect\X$, $F=\co(\overline{x}_{i})\oplus F_{i}$ and $F_{i}$ satisfies the following exact sequence
\[\eta_{i}: 0 \to
\co\oplus\co(\overline{x}_{j})\oplus\co(\overline{x}_{k})\to F_{i}
\to \mathcal{S}\to 0.
\]
It follows that
\[\zeta_{_{i}}^{\prime}: 0 \to E\to
E(\overline{x}_{i})\oplus(IE\backslash\co(\overline{x}_{i}))\to F_{i}
\to 0
\]
is also a distinguished exact sequence and $F_{i}$ is an indecomposable vector bundle.

Now we claim that $F_{i}$ is determined by the non-split sequence $\xi_{i,t}$. In fact, from the distinguished exact sequence $\gamma_{i}$, we get that the injective hull $I(E(\overline{x}_{i}))$ is a direct summand of $IF_{i}\oplus \co(\overline{x}_{i})$. In particular, for any $t\neq i$, $\co(\overline{x}_{i}+\overline{x}_{t})$ is a direct summand of $IF_{i}$. Moreover, applying $\Hom(-, \co(\overline{x}_{i}+\overline{x}_{t}))$ to $\eta_{i}$, we get $\Hom(F_{i}, \co(\overline{x}_{i}+\overline{x}_{t}))\cong k$. Thus, the nonzero morphism $\varphi_{t}: F_{i}\to \co(\overline{x}_{i}+\overline{x}_{t})$ is surjective. We only consider the case of $t=k$, the case of $t=j$ is similar. From the exact sequence
\[ 0 \to \co(\overline{x}_{j})\to
\co(\overline{x}_{i}+\overline{x}_{k})\to S_{j,p_{j}-2}^{(p_{j}-2)}
\to 0,
\]
we get $[\co(\overline{x}_{i}+\overline{x}_{k})]=[\co(\overline{x}_{j})]+[S_{j,p_{j}-2}^{(p_{j}-2)}]$.
It follows from $\eta_{i}$ that $[\ker \varphi_{k}]=[\co]+[\co(\overline{x}_{k})]+[S_{k,p_{k}-2}^{(p_{k}-2)}]=[\co(\overline{x}_{k})]+[\co((p_{k}-2)\vec{x}_{k})]=[E\langle(p_{k}-3)\vec{x}_{k}\rangle(\vec{x}_{k})]$.
Notice that $\Ext^{1}(\co(\overline{x}_{i}+\overline{x}_{k}),E\langle(p_{k}-3)\vec{x}_{k}\rangle(\vec{x}_{k}))\cong k$. Hence, to finish the proof of the claim, we only need to show that
$\ker \varphi_{k}$ is indecomposable. In fact, if $\ker \varphi_{k}=\co(\vec{y}_{1})\oplus \co(\vec{y}_{2})$, then applying $\Hom(\co(\overline{x}_{k}),-)$ to the exact sequence
\[\theta: 0 \to \co(\vec{y}_{1})\oplus \co(\vec{y}_{2})\to
 F_{i}\to \co(\overline{x}_{i}+\overline{x}_{k})
\to 0,
\]
we get $\Hom(\co(\overline{x}_{k}),\co(\vec{y}_{l}))\neq 0$ for $l=1$ or $2$, thus $\vec{y}_{l}-\overline{x}_{k}\geq 0$. On the other hand, applying $\Ext^{1}(-,\co(\vec{y}_{l}))$ to $\theta$, we get $\Ext^{1}(\co(\overline{x}_{i}+\overline{x}_{k}),\co(\vec{y}_{l}))\cong D\Hom(\co(\vec{y}_{l}),\co(\overline{x}_{i}+\overline{x}_{k}+\vec{\omega}))\neq 0.$ It follows that $\overline{x}_{i}+\overline{x}_{k}+\vec{\omega}-\vec{y}_{l}\geq 0.$ Hence, $\vec{y}_{l}-\overline{x}_{k}\leq \overline{x}_{i}+\vec{\omega}=\vec{\omega}+\vec{c}-\vec{x}_{j}-\vec{x}_{k}<\vec{\omega}+\vec{c},$ a contradiction. The claim is proved.

Moreover, from the distinguished exact sequence $\zeta_{i}$, we get that the projective cover $PF$ of $F$ is a direct summand of $P(E(\overline{x}_{i})\oplus IE)=P(E(\overline{x}_{i}))\oplus IE$, where $P(E(\overline{x}_{i}))=\co(\overline{x}_{i}+\vec{\omega})\oplus\co(\overline{x}_{i}-\vec{x}_{i})\oplus\co(\overline{x}_{i}-\vec{x}_{j})\oplus\co(\overline{x}_{i}-\vec{x}_{k})$.
From the fact that the shift functor preserves the rank, we get $\rk(PF_{i})=2\rk F_{i}=6.$ Notice that $\Hom(\co(\overline{x}_{i}-\vec{x}_{i}), F_{i})\cong k$, and each
morphism $f: \co(\overline{x}_{i}-\vec{x}_{i})\to F_{i}$ factors through the direct summand $\co(\overline{x}_{j})$ of $IE$. Hence, $\co(\overline{x}_{i}-\vec{x}_{i})$ is not a summand of $PF_{i}$. Thus, $PF_{i}=P(E(\overline{x}_{i})\oplus IE)\backslash (\co(\overline{x}_{i})\oplus \co(\overline{x}_{i}-\vec{x}_{i}))=\co(\overline{x}_{i}+\vec{\omega})\oplus\co(\overline{x}_{i}-\vec{x}_{j})\oplus\co(\overline{x}_{i}-\vec{x}_{k})
\oplus\co(\overline{x}_{j})\oplus\co(\overline{x}_{k})\oplus \co.$

In order to calculate the injective hull $IF_{i}$ of $F_{i}$, we consider the following pushout diagram
\[
\xymatrix{
\xi_{i}: 0 \ar[r] & E \langle(p_{k}-3)\vec{x}_{k}\rangle (\vec{x}_{k})\ar[d]\ar[r] & F_{i}\ar[d] \ar[r]\ar@{}[dr] &  \co(\overline{x}_{i}+\overline{x}_{k}) \ar@{=}[d]\ar[r] & 0  \\
 \ \ \ \ \ \ 0 \ar[r] & IE \langle(p_{k}-3)\vec{x}_{k}\rangle (\vec{x}_{k}) \ar[r] & F^{\prime}\ar[r] & \co(\overline{x}_{i}+\overline{x}_{k}) \ar[r]& 0.   \\
 }
\]
Recall that
$IE \langle(p_{k}-3)\vec{x}_{k}\rangle (\vec{x}_{k})=\co((p_{k}-2)\vec{x}_{k})\oplus \co(\vec{x}_{k}+\overline{x}_{i})\oplus \co(\vec{x}_{k}+\overline{x}_{j})\oplus \co(\vec{\omega}+(p_{k}-1)\vec{x}_{k}).$ It is easy to verify that $\Ext^{1}(\co(\overline{x}_{i}+\overline{x}_{k}),L)=0$ for any direct summand $L$ of $I(E \langle(p_{k}-3)\vec{x}_{k}\rangle (\vec{x}_{k}))$ different from $\co((p_{k}-2)\vec{x}_{k})$, that is, $ \co(\vec{x}_{k}+\overline{x}_{i})\oplus \co(\vec{x}_{k}+\overline{x}_{j})\oplus \co(\vec{\omega}+(p_{k}-1)\vec{x}_{k})$ is a direct summand of $F^{\prime}$. Moreover, the middle term of the non-split exact sequence
\[ 0 \to \co((p_{k}-2)\vec{x}_{k})\to
 E^{\prime}\to \co(\overline{x}_{i}+\overline{x}_{k})
\to 0
\]
has the expression $E^{\prime}=E\langle\vec{x}_{j}\rangle(\overline{x}_{i})$ by considering its class in $K_{0}(\mathbb{X})$:
$[E^{\prime}]=[\co((p_{k}-2)\vec{x}_{k})]+[\co(\overline{x}_{i}+\overline{x}_{k})]=[\co((p_{k}-1)\vec{x}_{k})]+[\co(\overline{x}_{i}+\vec{\omega})]=
[E\langle\vec{x}_{j}\rangle(\overline{x}_{i})]$.
Thus $F^{\prime}= \co(\vec{x}_{k}+\overline{x}_{i})\oplus \co(\vec{x}_{k}+\overline{x}_{j})\oplus \co(\vec{\omega}+(p_{k}-1)\vec{x}_{k})\oplus E\langle\vec{x}_{j}\rangle(\overline{x}_{i}).$ Since distinguished injectivity is preserved under taking pushout, we have $F_{i}\hookrightarrow F^{\prime}$ is a distinguished injection, which implies that $IF_{i}$ is a direct summand of $IF^{\prime}.$ Recall that $I(E\langle\vec{x}_{j}\rangle(\overline{x}_{i}))= \co(\overline{x}_{i}+\vec{x}_{j})\oplus \co(2\overline{x}_{i})\oplus \co(\overline{x}_{i}+\overline{x}_{k})\oplus \co(\overline{x}_{i}+\overline{x}_{j}+\vec{x}_{j}).$ It is easy to see that
$\Hom(\co(\vec{\omega}+(p_{k}-1)\vec{x}_{k}),\co(\overline{x}_{i}+\overline{x}_{j}+\vec{x}_{j}))\cong S_{\vec{x}_{j}}\cong k$ and
$\Hom(F_{i},\co(\overline{x}_{i}+\overline{x}_{j}+\vec{x}_{j}))\cong k$. Thus each morphism $f:F_{i}\to \co(\overline{x}_{i}+\overline{x}_{j}+\vec{x}_{j})$ factors through
$\co(\vec{\omega}+(p_{k}-1)\vec{x}_{k})$. We conclude that $\co(\overline{x}_{i}+\overline{x}_{j}+\vec{x}_{j})$ is not a summand of $IF_{i}$. Thus, in view of the rank, we have $IF_{i}=\bigoplus\limits_{t=1}^{3}(\co((p_{t}-1)\vec{x}_{t})\oplus \co(\overline{x}_{i}+\overline{x}_t)).
$
\end{proof}

\begin{proposition}\label{proposition2}
Assume $0\leq \vec{x}-\vec{x}_{j}-\vec{x}_{k}\leq \vec{x}\leq 2\vec{\omega}+\vec{c}$. Then there is a triangle in $\crd$:
\[\xi: G[-1]\to E\langle\vec{x}-\vec{x}_{j}\rangle\oplus E\langle\vec{x}-\vec{x}_{k}\rangle\to E\langle\vec{x}\rangle\to G,
\]
where $G$ is determined by the following non-split exact sequence in $\coh\X$:
\[\zeta: 0 \to E\langle\vec{x}-\vec{x}_{j}-\vec{x}_{k}\rangle[1]\to
 G\to \co(\vec{x})
\to 0.
\]
\end{proposition}

\begin{proof}
Write $\vec{x}=\sum\limits_{i=1}^{3}l_{i}\vec{x}_{i}$, by Proposition 4.20 in \cite{KLM2} there is a triangle in $\crd$
\[
\eta:  E\langle\vec{x}-\vec{x}_{j}\rangle\to E\langle\vec{x}\rangle\to E\langle\vec{x}-l_{j}\vec{x}_{j}\rangle(l_{j}\vec{x}_{j})\to
E\langle\vec{x}-\vec{x}_{j}\rangle[1].
\]
Then $\eta$ induces the following homotopy pullback commutative diagram(\cite{H})
\[
\xymatrix{
G[-1] \ar[d]\ar[r] & E\langle\vec{x}-\vec{x}_{k}\rangle\ar[d]\ar[r] & E\langle\vec{x}-l_{j}\vec{x}_{j}\rangle(l_{j}\vec{x}_{j})\ar@{=}[d]\ar[r] & G\ar[d] \\
 E\langle\vec{x}-\vec{x}_{j}\rangle\ar[r] & E \langle\vec{x}\rangle  \ar[r] & E\langle\vec{x}-l_{j}\vec{x}_{j}\rangle(l_{j}\vec{x}_{j}) \ar[r]& E\langle\vec{x}-\vec{x}_{j}\rangle[1].   \\
 }
\]
It follows that
\[\xi: G[-1]\to E\langle\vec{x}-\vec{x}_{j}\rangle\oplus E\langle\vec{x}-\vec{x}_{k}\rangle\to E\langle\vec{x}\rangle\to G
\]
is a triangle in $\crd$.

Now we claim that $G$ is determined by the exact sequence $\zeta$. In fact, there is an exact sequence in $\coh \mathbb{X} $
\[\theta: 0 \to E\langle\vec{x}-\vec{x}_{k}\rangle\to E\langle\vec{x}-l_{j}\vec{x}_{j}\rangle(l_{j}\vec{x}_{j})\to \mathcal{S}
\to 0,
\]
where by calculating the determinant $\det \mathcal{S}=l_{j}\vec{x}_{j}+\vec{x}_{k}$. We conclude that $\mathcal{S}=S_{j,l_{j}-1}^{(l_{j})}\oplus S_{k,l_{k}}$. Notice that $\theta$ induces the pushout commutative diagram
\[
\xymatrix{
0\ar[r]& E\langle\vec{x}-\vec{x}_{k}\rangle\ar[d]\ar[r] & E\langle\vec{x}-l_{j}\vec{x}_{j}\rangle(l_{j}\vec{x}_{j})\ar[d]\ar[r] & \mathcal{S}\ar@{=}[d]\ar[r]& 0 \\
 0\ar[r]& I(E\langle\vec{x}-\vec{x}_{k}\rangle)\ar[r]&G^{\prime} \ar[r]&\mathcal{S}\ar[r]&0,  \\
 }
\]
where $I(E\langle\vec{x}-\vec{x}_{k}\rangle)= \co(\vec{x}-\vec{x}_{k})\oplus \co(\vec{\omega}+(l_{i}+1)\vec{x}_{i})\oplus \co(\vec{\omega}+(l_{j}+1)\vec{x}_{j})\oplus \co(\vec{\omega}+l_{k}\vec{x}_{k})$ is the injective hull of $E\langle\vec{x}-\vec{x}_{k}\rangle.$ It is easy to find that only one direct summand,
$\co(\vec{\omega}+(l_{j}+1)\vec{x}_{j})$, of $I(E\langle\vec{x}-\vec{x}_{k}\rangle)$ vanishes under the functor $\Ext^{1}(\mathcal{S}, -).$ Thus $\co(\vec{\omega}+(l_{j}+1)\vec{x}_{j})$ is a direct summand of $G^{\prime}$. By using similar arguments as shown in Proposition \ref{proposition1}, we get that $G^{\prime}=\co(\vec{\omega}+(l_{j}+1)\vec{x}_{j})\oplus G$ for an indecomposable sheaf $G$.
From the distinguished injection $E\langle\vec{x}-l_{j}\vec{x}_{j}\rangle(l_{j}\vec{x}_{j})\hookrightarrow G^{\prime},$ we get that the injective hull of $E\langle\vec{x}-l_{j}\vec{x}_{j}\rangle(l_{j}\vec{x}_{j})$ is a direct summand of $IG^{\prime}$. In particular, $\co(\vec{x})$ is a direct summand of $IG,$ which induces an epimorphism $\varphi: G\to \co(\vec{x})$.
We claim that $\ker \varphi$ is indecomposable. Otherwise, $\ker \varphi=\co(\vec{y}_{1}) \oplus \co(\vec{y}_{2})$ for some $\vec{y}_{1}, \vec{y}_{2}\in \mathbb{L}$.
Applying $\Hom(\co(\vec{\omega}+(l_{i}+1)\vec{x}_{i}), -)$ to the exact sequence
\[\theta^{\prime}: 0 \to \co(\vec{y}_{1}) \oplus \co(\vec{y}_{2})\to G\to \co(\vec{x})
\to 0,
\]
we get $\Hom(\co(\vec{\omega}+(l_{i}+1)\vec{x}_{i}), \co(\vec{y}_{t}))\neq 0$ for some $t=1,2$. Hence, $\vec{y}_{t}\geq \vec{\omega}+(l_{i}+1)\vec{x}_{i}.$
On the other hand, applying $\Ext^{1}(-, \co(\vec{y}_{t}))$ to $\theta^{\prime}$, we get
$\Ext^{1}(\co(\vec{x}), \co(\vec{y}_{t}))=D\Hom(\co(\vec{y}_{t}), \co(\vec{x}+\vec{\omega}))\neq 0.$ Thus $\vec{\omega}+(l_{i}+1)\vec{x}_{i}\leq \vec{y}_{t}\leq \vec{x}+\vec{\omega},$ a contradiction, as claimed. Therefore, $\ker \varphi$ is an extension bundle, which is determined by its class in $K_0(\X)$. Notice that
$[\ker \varphi]=[G]-[\co(\vec{x})]=[\co(\vec{\omega}+(l_{i}+1)\vec{x}_{i})]+[\co(\vec{\omega}+l_{k}\vec{x}_{k})]+[S_{j,l_{j}-1}^{(l_{j})}]
=[\co(\vec{\omega}+(l_{i}+1)\vec{x}_{i}+l_{j}\vec{x}_{j})]+[\co(\vec{\omega}+l_{k}\vec{x}_{k})]=[E\langle l_{i}\vec{x}_{i}+(l_{j}-1)\vec{x}_{j}+(p_{k}-1-l_{k})\vec{x}_{k}\rangle(l_{k}\vec{x}_{k})]=[E\langle \vec{x}-\vec{x}_{j}-\vec{x}_{k}\rangle[1]].$
We obtain the exact sequence $\xi$. Moreover, the fact that $\Ext^{1}(\co(\vec{x}), E\langle l_{i}\vec{x}_{i}+(l_{j}-1)\vec{x}_{j}+(p_{k}-1-l_{k})\vec{x}_{k}\rangle(l_{k}\vec{x}_{k}))\cong k$ implies that the middle term $G$ of $\xi$ is uniquely determined.
\end{proof}

By using  the triangle
\[ E\langle\vec{x}\rangle\to E\langle\vec{x}+\vec{x}_{j}\rangle\to
 E\langle\vec{x}-l_{j}\vec{x}_{j}\rangle((l_{j}+1)\vec{x}_{j})\to E\langle\vec{x}\rangle[1]
\]
and the induced homotopy pushout commutative diagram
\[
\xymatrix{
\overline{E}\ar@{=}[d]\ar[r]& E\langle\vec{x}\rangle\ar[d]\ar[r] & E\langle\vec{x}+\vec{x}_{j}\rangle\ar[d]\ar[r] & E\langle\vec{x}-l_{j}\vec{x}_{j}\rangle((l_{j}+1)\vec{x}_{j})\ar@{=}[d] \\
\overline{E}\ar[r]& E\langle\vec{x}+\vec{x}_{k}\rangle\ar[r]&H \ar[r]& E\langle\vec{x}-l_{j}\vec{x}_{j}\rangle((l_{j}+1)\vec{x}_{j}),  \\
 }
\]
where $\overline{E}=E\langle \vec{x}+(p_{j}-2-l_{j})\vec{x}_{j}\rangle((l_{j}+2-p_{j})\vec{x}_{j}),$
we have the following similar result.

\begin{proposition}\label{proposition3}
Assume $0\leq \vec{x}\leq \vec{x}+\vec{x}_{j}+\vec{x}_{k}\leq 2\vec{\omega}+\vec{c}$. Then there is a triangle in $\crd$:
\[H[-1]\to E\langle\vec{x}\rangle \to E\langle\vec{x}+\vec{x}_{j}\rangle\oplus E\langle\vec{x}+\vec{x}_{k}\rangle\to H,
\]
where $H$ is determined by the following non-split exact sequence:
\[ 0 \to E\langle\vec{x}+\vec{x}_{j}\rangle\to
 H\to \co(\vec{x}+\vec{x}_{k})
\to 0.
\]
\end{proposition}

\section{Tubular tilting objects in $\underline{\vect}\mathbb{X}$}

We focus on the tubular tilting objects in the stable category of vector bundles $\crd:=\underline{\vect}\mathbb{X}$ of genus one. In view of the work on \cite{CLR}, we only  consider the weighted projective line of genus one and weight triple, that is, $\X$ is of weight type $(2,4,4), (2,3,6)$ or $(3,3,3)$.
We start with studying the properties of tilting objects in $\crd$.

\begin{lemma}\label{lemma4} Assume that $E\oplus F$ is extension-free in $\crd$ and $\mu F=\mu(E[1])$. Then $F\oplus(\tau^{-1}E[1])$ is extension-free if and only if $\crd(F, E(\vec{c}-\vec{\omega}))=0.$
\end{lemma}

\begin{proof} Notice that $\mu F=\mu(E[1])$, by (\cite{KLM2}, Theorem A.6) and the semi-stability of vector bundles, we have
$\crd(\tau^{-1}E[1], F[n])=0=\crd(F,(\tau^{-1}E[1])[n])$ for any $n\neq 0,1$. Moreover, the assumption that $E\oplus F$ is extension-free implies that
$\crd(\tau^{-1}E[1], F[1])\cong D\crd(F, E[1])=0$. Thus, the result is easily obtained since
$\crd(F, (\tau^{-1}E[1])[1])=\crd(F, E(\vec{c}-\vec{\omega})).$
\end{proof}

\begin{proposition}\label{proposition5}
Let $T=\bigoplus\limits_{ i\in I}T_{i}$ be a tilting object in $\crd$ with indecomposable direct summand $T_{i}\in \crd[a, \alpha(a)]$ for some $a\in \mathbb{Q}$.
Assume that $I_{1}=\{i\in I|\mu T_{i}=a\}$ and $I_{2}=\{i\in I|\mu T_{i}=\alpha(a)\}.$ Then $T^{\prime}=(\bigoplus\limits_{i\in I_{1}}\tau^{-1}T_{i}[1])\oplus(\bigoplus\limits_{j\in I\backslash I_{1}}T_{j})$ is a tilting object in $\crd$ if and only if $\crd(T_{i_{2}}, T_{i_{1}}(\vec{c}-\vec{\omega}))=0$ for any $i_{1}\in I_{1}$ and $i_{2}\in I_{2}.$
\end{proposition}

\begin{proof}
Notice that each direct summand of $T^{\prime}$ belongs to $\crd(a, \alpha(a)]$. Hence by (\cite{KLM2}, Theorem A.6) and the semi-stability of vector bundles,
for any $i\in I_{1}$ and $t\in I\backslash(I_{1}\cup I_{2})$, we have $\crd(T_{t}, (\tau^{-1}T_{i}[1])[n])=0$ for any $n\neq 0$
and $\crd(\tau^{-1}T_{i}[1], T_{t}[n])=0$ for any $n\neq 1$.
Moreover, if $n=1$, then by Serre duality, we have $\crd(\tau^{-1}T_{i}[1], T_{t}[1])=D\crd(T_{t}, T_{i}[1])=0$.
Thus, $T^{\prime}$ is extension-free if and only if $\tau^{-1}T_{i_{1}}[1]\oplus T_{i_{2}}$ is extension-free for any $i_{1}\in I_{1}$ and $i_{2}\in I_{2}$, that is,
$\crd(T_{i_{2}}, T_{i_{1}}(\vec{c}-\vec{\omega}))=0$ by the preceding lemma.
Notice that the direct summand of $T^{\prime}$ can be arranged as a complete exceptional sequence, thus $T^{\prime}$ is automatically a tilting object.
\end{proof}

Since $T_{\cub}$ is a tilting object in $\crd$, the direct summands of $T_{\cub}$ belong to $\crd[0, \alpha(0)]$, and only one of them has minimal (respectively, maximal) slope, that is, $\mu E=0$ and $\mu(E\langle 2\vec{\omega}+\vec{c} \rangle)=\frac{\delta(\vec{c})}{2}=\alpha(0)$. Then by Proposition \ref{proposition5}, we have the following results.

\begin{corollary}\label{corollary6}
Let $T^{\prime}=\tau^{-1}E[1]\oplus(\bigoplus\limits_{0<\vec{x}\leq 2\vec{\omega}+\vec{c}}E\langle\vec{x}\rangle)$.
\begin{itemize}
\item[(1)] If weight type of $\X$ is $(2,3,6)$ or $(2,4,4)$, then $T^{\prime}$ is a tilting object in $\crd$.
\item[(2)] If weight type of $\X$ is $(3,3,3)$, then $T^{\prime}$ is not a tilting object in $\crd$.
\end{itemize}
 \end{corollary}

\begin{proof}
(1) If $\X$ has weight type $(2,3,6)$ or $(2,4,4)$, then by (\cite{KLM2}, Proposition 4.15), $E\langle 2\vec{\omega}+\vec{c} \rangle=E((p_{2}-1)\vec{x}_{2}+(p_{3}-1)\vec{x}_{3}-\vec{c})$. Thus
$\crd(E\langle 2\vec{\omega}+\vec{c} \rangle, E(\vec{c}-\vec{\omega}))=\crd(E, E(\vec{x}_{1}+2\vec{x}_{2}+2\vec{x}_{3}-\vec{c}))=0$ since
$\vec{x}_{1}+2\vec{x}_{2}+2\vec{x}_{3}-\vec{c}\neq 0$ or $\overline{x}_{i}$ for $1\leq i\leq 3$. Then by Proposition \ref{proposition5}, $T^{\prime}$ is a tilting object in
$\crd$.

(2) If $\X$ has weight type $(3,3,3)$, then by (\cite{KLM2}, Proposition 4.15),
$E\langle 2\vec{\omega}+\vec{c} \rangle=E\langle\vec{x}_{1}+\vec{x}_{2}+\vec{x}_{3}\rangle=(E[1])(\vec{\omega})$. Thus
$\crd(E\langle 2\vec{\omega}+\vec{c} \rangle, E(\vec{c}-\vec{\omega}))=\crd(\tau E[1], E(\vec{c}-\vec{\omega}))=
\crd(\tau E[2], E(\vec{c}-\vec{\omega})[1])=D\crd(E(\vec{c}-\vec{\omega}), E(2\vec{\omega}+\vec{c}))=D\crd(E, E)\neq 0.$
By Proposition \ref{proposition5}, $T^{\prime}$ is not a tilting object in
$\crd$.
\end{proof}

Now we begin to construct a tubular tilting object in $\crd$ case by case. Firstly, we consider $\X$ of weight type $(2,4,4)$. The main strategy is to use the cluster mutation. More precisely, we first construct an original cluster tilting object in $\crc$ from $T_{\cub}$, then by a sequence of tubular mutations based on Keller's soft-ware we deduce a tilting object in $\crd$ with tubular endomorphism algebra.

\begin{theorem}\label{theorem7}
Assume that $\X$ has weight type $(2,4,4)$. Then there exists a tilting object in $\crd$:
$T_{(2,4,4)}=\tau^{-1}(E^{\ast}\oplus E\langle\vec{x}_{2}\rangle \oplus E\langle\vec{x}_{3}\rangle \oplus E(-\vec{\omega}))[1]\oplus
(E\langle 2\vec{x}_{2}+\vec{x}_{3}\rangle \oplus E\langle \vec{x}_{2}+2\vec{x}_{3}\rangle \oplus E(\vec{x}_{1}-\vec{x}_{2}+\vec{x}_{3})\oplus E(\vec{x}_{1}+\vec{x}_{2}-\vec{x}_{3})\oplus E\langle 2\vec{x}_{2}+2\vec{x}_{3}\rangle^{\ast})$, whose endomorphism algebra is a tubular algebra of type $(2,4,4)$,
where $E^{\ast}$ and $E\langle 2\vec{x}_{2}+2\vec{x}_{3}\rangle^{\ast}$ are determined by the following exact sequences:
\[ 0\to E\langle\vec{x}_{3}\rangle\to E^{\ast}\to \co(\vec{x}_{2})\to 0,
\]
and
\[ 0\to E\langle\vec{x}_{2}+\vec{x}_{3}\rangle\to E\langle 2\vec{x}_{2}+2\vec{x}_{3}\rangle^{\ast}\to \co(2\vec{x}_{2}+2\vec{x}_{3}-\vec{x}_{1})\to 0.
\]
 \end{theorem}

\begin{proof}
By Corollary \ref{corollary6}, we get that $T^{\prime}=(\tau^{-1}E[1])\oplus(\bigoplus\limits_{0<\vec{x}\leq 2\vec{\omega}+\vec{c}}E\langle\vec{x}\rangle)$ is a tilting object in $\crd$, and each direct summand of $T^{\prime}$ belongs to $\crd(0, \alpha(0)].$ Hence by Lemma \ref{lemma9},
$T^{\prime}$ is a cluster tilting object in $\crc$. Notice that $T^{\prime}=T_{\cub}$ in $\crc$, the Gabriel quiver of the endomorphism algebra $\crc(T^{\prime}, T^{\prime})$ has the form:
\[
\xymatrix{
E\ar[r]\ar[d]&E\langle\vec{x}_{2}\rangle \ar[r]\ar[d]& E\langle 2\vec{x}_{2}\rangle \ar[d] \\
 E\langle\vec{x}_{3}\rangle\ar[r]\ar[d]&E\langle\vec{x}_{2}+\vec{x}_{3}\rangle\ar[lu]\ar[r]\ar[d]&E\langle 2\vec{x}_{2}+\vec{x}_{3}\rangle\ar[lu]\ar[d]\\
  E\langle2\vec{x}_{3}\rangle\ar[r]&E\langle\vec{x}_{2}+2\vec{x}_{3}\rangle\ar[lu]\ar[r]&E\langle 2\vec{x}_{2}+2\vec{x}_{3}\rangle.\ar[lu].
 }
\]
Using Keller's soft-ware, we know that under taking quiver mutations $u_{1}u_{2}u_{3}u_{4}u_{5}$ for the quiver
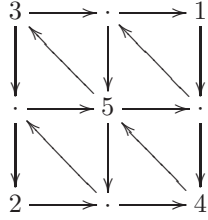
\[
\xymatrix{
3\ar[r]\ar[d]&\cdot \ar[r]\ar[d]& 1 \ar[d] \\
\cdot\ar[r]\ar[d]&5\ar[lu]\ar[r]\ar[d]&\cdot\ar[lu]\ar[d]\\
 2\ar[r]&\cdot\ar[lu]\ar[r]&4\ar[lu],
 }
\]
we can obtain the following quiver
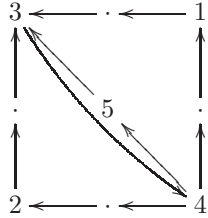
\[
\xymatrix{
3\ar@/_/[rrdd]&\cdot \ar[l]& 1\ar[l]  \\
\cdot\ar[u]&5\ar[lu]&\cdot\ar[u]\\
 2\ar[u]&\cdot\ar[l]&4\ar[lu]\ar[l]\ar[u].
 }
\]
By Lemma \ref{theorem BIRS}, the quiver mutation $u_{i}$ corresponds to the cluster tilting mutation in the triangulated category $\crc$ at the corresponding indecomposable direct summand of $T'.$ In more details, the cluster tilting mutation corresponding to $u_{i}$ for $1\leq i\leq 5$ is given as follows.

(1) $u_{1}$ corresponds to the following triangle
\[ E\langle2\vec{x}_{2}\rangle\to E\langle 2\vec{x}_{2}+\vec{x}_{3}\rangle\to E\langle 2\vec{x}_{2}\rangle^{\ast}\to  E\langle2\vec{x}_{2}\rangle[1],
\]
where $E\langle 2\vec{x}_{2}\rangle^{\ast}=E\langle2\vec{x}_{2}\rangle(\vec{x}_{3})=E(\vec{x}_{1}-\vec{x}_{2}+\vec{x}_{3})$ by Lemma \ref{lemma1};

(2) $u_{2}$ corresponds to the following triangle
\[ E\langle2\vec{x}_{3}\rangle\to E\langle \vec{x}_{2}+2\vec{x}_{3}\rangle\to E\langle 2\vec{x}_{3}\rangle^{\ast}\to  E\langle2\vec{x}_{3}\rangle[1],
\]
where $E\langle 2\vec{x}_{3}\rangle^{\ast}=E\langle2\vec{x}_{3}\rangle(\vec{x}_{2})=E(\vec{x}_{1}+\vec{x}_{2}-\vec{x}_{3})$ by Lemma \ref{lemma1};

(3) $u_{3}$ corresponds to the following triangle
\[ E \to E\langle \vec{x}_{2}\rangle\oplus E\langle \vec{x}_{3}\rangle\to E^{\ast}\to  E[1],
\]

where $E^{\ast}$ is determined by the non-split exact sequence
\begin{equation}\label{equation1} 0\to E\langle \vec{x}_{3}\rangle\to E^{\ast}\to \co(\vec{x}_{2})\to 0;
\end{equation}

(4) $u_{4}$ corresponds to the following triangle
\[ E\langle2\vec{x}_{2}+2\vec{x}_{3}\rangle^{\ast} \to E\langle \vec{x}_{2}+2\vec{x}_{3}\rangle\oplus E\langle 2\vec{x}_{2}+\vec{x}_{3}\rangle\to E\langle 2\vec{x}_{2}+2\vec{x}_{3}\rangle,
\]
where $E\langle2\vec{x}_{2}+2\vec{x}_{3}\rangle^{\ast}$ is determined by the non-split exact sequence
\[ 0\to E\langle \vec{x}_{2}+\vec{x}_{3}\rangle\to E\langle 2\vec{x}_{2}+2\vec{x}_{3}\rangle^{\ast}\to \co(2\vec{x}_{2}+2\vec{x}_{3}-\vec{x}_{1})\to 0;
\]

(5) $u_{5}$ corresponds to the following triangle
\begin{equation}\label{equation2} E\langle \vec{x}_{2}+\vec{x}_{3}\rangle[-1] \to E\langle \vec{x}_{2}+\vec{x}_{3}\rangle^{\ast}\to E^{\ast}\xrightarrow{\varphi} E\langle \vec{x}_{2}+\vec{x}_{3}\rangle.
\end{equation}
Now we claim that $ E\langle \vec{x}_{2}+\vec{x}_{3}\rangle^{\ast}=E(-\vec{\omega}).$ In fact, from (\ref{equation1}), we have $\mu(E^{\ast})=\frac{2}{3}$,
then by comparing the slopes of $E^{\ast}$ and $E\langle \vec{x}_{2}+\vec{x}_{3}\rangle$, we obtain that the nonzero morphism $\varphi$ in $\coh\X$ is surjective.
Moreover, $\det(\ker \varphi)=\det(E^{\ast})-\det(E\langle \vec{x}_{2}+\vec{x}_{3}\rangle)=0$ implies that $\ker \varphi=0.$ Thus we get an exact sequence
\begin{equation}\label{equation3}0 \to \co \to E^{\ast}\xrightarrow{\varphi} E\langle \vec{x}_{2}+\vec{x}_{3}\rangle \to 0.
\end{equation}
Taking pullback along (\ref{equation3}), we get the following commutative diagram
\begin{equation}\label{equation4}
\xymatrix{
0\ar[r]&\co\ar@{=}[d]\ar[r]&F\ar[d]\ar[r] & P(E\langle \vec{x}_{2}+\vec{x}_{3}\rangle) \ar[d]\ar[r] & 0 \\
0\ar[r]&\co\ar[r]& E^{\ast}\ar[r]^{\varphi \ \ \ }&E\langle \vec{x}_{2}+\vec{x}_{3}\rangle\ar[r]& 0.
 }
\end{equation}
Recall that $P(E\langle \vec{x}_{2}+\vec{x}_{3}\rangle)=\co(\vec{\omega})\oplus \co(\vec{x}_{2}-\vec{x}_{3})\oplus \co(\vec{x}_{3}-\vec{x}_{2})\oplus \co(-\vec{\omega}).$ It is easy to check that all the indecomposable summands but $\co(-\vec{\omega})$ of $P(E\langle \vec{x}_{2}+\vec{x}_{3}\rangle)$ vanish under
the functor $\Ext^{1}(-, \co).$ That is, the other three direct summands of $P(E\langle \vec{x}_{2}+\vec{x}_{3}\rangle)$ are direct summands of $F$. It follows that
$F=\co(\vec{\omega})\oplus \co(\vec{x}_{2}-\vec{x}_{3})\oplus \co(\vec{x}_{3}-\vec{x}_{2})\oplus E(-\vec{\omega})$. Under factoring line bundle summands in the
distinguished exact sequence
\[ 0\to F \to E^{\ast}\oplus P(E\langle \vec{x}_{2}+\vec{x}_{3}\rangle)\to E\langle \vec{x}_{2}+\vec{x}_{3}\rangle\to 0,
\]
we get a triangle
\[  E\langle \vec{x}_{2}+\vec{x}_{3}\rangle[-1]\to E(-\vec{\omega}) \to E^{\ast}\xrightarrow{\varphi}  E\langle \vec{x}_{2}+\vec{x}_{3}\rangle.
\]
Comparing with (\ref{equation2}), we get $E\langle \vec{x}_{2}+\vec{x}_{3}\rangle^{\ast}=E(-\vec{\omega}),$ as claimed.

Therefore, we obtain a cluster tilting object in $\crc$ with endomorphism algebra of the Gabriel quiver
\[
\xymatrix{
E^{\ast}\ar@/_/[rrdd]&E\langle\vec{x}_{2}\rangle \ar[l]& E(\vec{x}_{1}-\vec{x}_{2}+\vec{x}_{3}) \ar[l] \\
 E\langle\vec{x}_{3}\rangle\ar[u]&E(-\vec{\omega})\ar[lu]&E\langle 2\vec{x}_{2}+\vec{x}_{3}\rangle\ar[u]\\
  E(\vec{x}_{1}+\vec{x}_{2}-\vec{x}_{3})\ar[u]&E\langle\vec{x}_{2}+2\vec{x}_{3}\rangle\ar[l]&E\langle 2\vec{x}_{2}+2\vec{x}_{3}\rangle^{\ast}\ar[lu]\ar[l]\ar[u].
 }
\]
Let $T_{(2,4,4)}=\tau^{-1}(E^{\ast}\oplus E\langle\vec{x}_{2}\rangle \oplus E\langle\vec{x}_{3}\rangle \oplus E(-\vec{\omega}))[1]\oplus
(E(\vec{x}_{1}-\vec{x}_{2}+\vec{x}_{3})\oplus E(\vec{x}_{1}+\vec{x}_{2}-\vec{x}_{3})\oplus E\langle 2\vec{x}_{2}+\vec{x}_{3}\rangle \oplus E\langle \vec{x}_{2}+2\vec{x}_{3}\rangle \oplus  E\langle 2\vec{x}_{2}+2\vec{x}_{3}\rangle^{\ast})$.
Then $T_{(2,4,4)}$ is a cluster tilting object in $\crc$ with each direct summand belongs to the interval category $\crd(\frac{2}{3}, \alpha(\frac{2}{3})]$.
Hence by Lemma \ref{lemma9}, $T_{(2,4,4)}$ is a tilting object in $\crd$ with endomorphism algebra:
\[
\xymatrix{
\tau^{-1}E^{\ast}[1]&\tau^{-1}E\langle\vec{x}_{2}\rangle[1] \ar[l]& E(\vec{x}_{1}-\vec{x}_{2}+\vec{x}_{3}) \ar[l] \\
 \tau^{-1}E\langle\vec{x}_{3}\rangle[1]\ar[u]&\tau^{-1}E(-\vec{\omega})[1]\ar[lu]&E\langle 2\vec{x}_{2}+\vec{x}_{3}\rangle\ar[u]\\
  E(\vec{x}_{1}+\vec{x}_{2}-\vec{x}_{3})\ar[u]&E\langle\vec{x}_{2}+2\vec{x}_{3}\rangle\ar[l]&E\langle 2\vec{x}_{2}+2\vec{x}_{3}\rangle^{\ast},\ar[lu]\ar[l]\ar[u]
 }
\]
which is a tubular algebra of type $(2,4,4).$
\end{proof}

Next, we consider $\X$ of weight type $(2,3,6)$. Similarly, we start with the cluster tilting object $T_{\cub}$ and the Gabriel quiver of its endomorphism algebra. It is more complicated than  the weight type $(2,4,4)$ to obtain the desired tilting object by using the cluster mutation. We need use
the mutations of the quiver at some vertex twice. We omit the process of describing the corresponding cluster tilting mutations to quiver mutations in the proof.

\begin{theorem}\label{theorem24}
Assume that $\X$ has weight type $(2,3,6)$. Then there exists a tilting object $T_{(2,3,6)}$ in $\crd$, whose endomorphism algebra is a tubular
algebra of type $(2,3,6)$, here $T_{(2,3,6)}=\tau H[-1]\oplus E\langle 3\vec{x}_{3}\rangle \oplus E(3\vec{x}_{3})\oplus E(2\vec{x}_{2}-2\vec{x}_{3})
\oplus E\langle 4\vec{x}_{3}\rangle^{\ast\ast}\oplus E\langle \vec{x}_{2}+\vec{x}_{3}\rangle\oplus E\langle \vec{x}_{2}\rangle(\vec{x}_{3})\oplus \tau^{-1}E[1]\oplus \tau^{-1}E\langle \vec{x}_{3}\rangle[1]\oplus \tau^{-1}E\langle\vec{x}_{2}+2\vec{x}_{3}\rangle^{\ast}[1]$ with the Gabriel quiver of the endomorphism algebra given as follows:
\begin{equation}\label{equation37}
{\tiny
{\xymatrix@-1.3pc{
&E\langle 3\vec{x}_{3}\rangle\ar[rrrr] & & &&E(3\vec{x}_{3})\ar[rd] & \\
\tau H[-1]\ar[ru]\ar[rd]\ar[rrr]&&&E(2\vec{x}_{2}-2\vec{x}_{3})\ar[rrr]&&& \tau^{-1}E\langle\vec{x}_{2}+2\vec{x}_{3}\rangle^{\ast}[1],\\
&E\langle 4\vec{x}_{3}\rangle^{\ast\ast}\ar[r]& E\langle \vec{x}_{2}+\vec{x}_{3}\rangle\ar[r]&E\langle \vec{x}_{2}\rangle(\vec{x}_{3})\ar[r]&\tau^{-1}E[1]\ar[r]&\tau^{-1}E\langle \vec{x}_{3}\rangle[1]\ar[ru]\\
 }}}
\end{equation}
where  $E\langle\vec{x}_{2}+2\vec{x}_{3}\rangle^{\ast}$ is determined by the exact sequence
\[0\to E\langle \vec{x}_{3}\rangle \to E\langle\vec{x}_{2}+2\vec{x}_{3}\rangle^{\ast}\to \co(\vec{x}_{2}+2\vec{x}_{3}-\vec{x}_{1}) \to 0,\]
$E\langle 4\vec{x}_{3}\rangle^{\ast\ast}$ is determined by the exact sequence
\[
0\to E\langle 2\vec{x}_{3}\rangle(\vec{\omega}) \to E\langle 4\vec{x}_{3}\rangle^{\ast\ast}\to \co(\vec{\omega}+\vec{x}_{2}) \to 0,
\]
$H$ is determined by the exact sequence
\begin{equation}
0\to E\langle 4\vec{x}_{3}\rangle (\vec{x}_{2})\to H\oplus\co(\vec{x}_{1}+\vec{x}_{2}+\vec{x}_{3})\to G \to 0,
\end{equation}
and $G$ is determined by the exact sequence
\[
0\to E\langle 2\vec{x}_{3}\rangle (\vec{x}_{1})\to G\to \co(\vec{x}_{1}+ \vec{x}_{2}) \to 0.
\]
\end{theorem}

\begin{proof}
By Corollary \ref{corollary6}, we get that $T^{\prime}=\tau^{-1}E[1]\oplus(\bigoplus\limits_{0<\vec{x}\leq 2\vec{\omega}+\vec{c}}E\langle\vec{x}\rangle)$ is a tilting object in $\crd$, and each direct summand of $T^{\prime}$ belongs to $\crd(0, \alpha(0)]$. Hence, by Lemma \ref{lemma9}, $T^{\prime}=T_{\cub}$
in $\crc$ is a cluster tilting object. The Gabriel quiver of the endomorphism algebra $\crc(T^{\prime}, T^{\prime})$ has the form:
\[
\xymatrix{
E\ar[r]\ar[d]&E\langle\vec{x}_{3}\rangle\ar[r]\ar[d] & E\langle2\vec{x}_{3}\rangle\ar[r]\ar[d] &E\langle3\vec{x}_{3}\rangle\ar[r]\ar[d]&E\langle4\vec{x}_{3}\rangle\ar[d]   \\
E\langle\vec{x}_{2}\rangle\ar[r]&E\langle\vec{x}_{2}+\vec{x}_{3}\rangle\ar[r]\ar[lu] & E\langle\vec{x}_{2}+2\vec{x}_{3}\rangle\ar[r]\ar[lu] &E\langle\vec{x}_{2}+3\vec{x}_{3}\rangle\ar[r]\ar[lu]&E\langle\vec{x}_{2}+4\vec{x}_{3}\rangle.\ar[lu]\\
 }
\]
Using Keller's soft-ware, we know that under taking quiver mutations $u_{1}u_{2}u_{3}u_{4}u_{5}u_{6}u_{1}$ for the quiver
\[
\xymatrix{
a\ar[r]\ar[d]&b\ar[r]\ar[d] & 5\ar[r]\ar[d] &c\ar[r]\ar[d]&1\ar[d]   \\
3\ar[r]&d\ar[r]\ar[lu] & 4\ar[r]\ar[lu] &6\ar[r]\ar[lu]&2\ar[lu],\\
 }
\]
we can obtain the following quiver
\[
\xymatrix{
&c\ar[rrrr] & & &&5\ar[rd] & \\
6\ar[ru]\ar[rd]\ar[rrr]&&&2\ar[rrr]&&&4\ar@/^/[llllll]\\
&1\ar[r]& d\ar[r]&3\ar[r]&a\ar[r]&b\ar[ru]\\
 }
 \]
By using the relationship between the quiver mutation and the cluster tilting mutation, and the similar discussion to Theorem \ref{theorem7}, we get that
$T_{(2,3,6)}=\tau H[-1]\oplus E\langle 3\vec{x}_{3}\rangle \oplus E(3\vec{x}_{3})\oplus E(2\vec{x}_{2}-2\vec{x}_{3})
\oplus E\langle 4\vec{x}_{3}\rangle^{\ast\ast}\oplus E\langle \vec{x}_{2}+\vec{x}_{3}\rangle\oplus E\langle \vec{x}_{2}\rangle(\vec{x}_{3})\oplus \tau^{-1}E[1]\oplus \tau^{-1}E\langle \vec{x}_{3}\rangle[1]\oplus \tau^{-1}E\langle\vec{x}_{2}+2\vec{x}_{3}\rangle^{\ast}[1]$ is a cluster tilting object
in $\crc$, each of whose indecomposable direct summands belongs to $\crd(\alpha^{-1}(\frac{11}{3}), \frac{11}{3}].$ Therefore $T_{(2,3,6)}$  is a tilting object
in $\crd$ with endomorphism algebra (\ref{equation37}), which is a tubular algebra of type $(2,3,6).$
\end{proof}

In the rest of this section, we are devoted to the only left case, $\X$ of weight type $(3,3,3)$.
Since  $T_{\cub}$ is not a cluster tilting object in the cluster category, we try to construct a replacement in $\crd$ such that it is a cluster tilting object, then use the cluster mutation to get what we want.

\begin{theorem}\label{theorem8}
Assume that $\X$ has weight type $(3,3,3)$. Then there exists a tilting object $T_{(3,3,3)}=E\oplus(\bigoplus\limits_{i=1}^{3}E\langle\vec{x}_{i}\rangle)\oplus
(\bigoplus\limits_{i=1}^{3}F_{i}[-1])\oplus G$ in $\crd$, whose endomorphism algebra is a tubular
algebra of type $(3,3,3)$, here $F_{i}$ is determined by
\begin{equation}\label{equation5} 0 \to E\langle2\vec{x}_{i}\rangle\to F_{i}\to \co(\vec{x}_{1}+\vec{x}_{2}+\vec{x}_{3})\to 0
\end{equation}
and $G$ is determined by the exact sequence for each $1\leq i\leq 3$:
\begin{equation}\label{equation6} 0 \to F_{i}\to G[1]\to E\langle\vec{x}_{1}+\vec{x}_{2}+\vec{x}_{3}\rangle(\vec{x}_{i})\to 0.
\end{equation}
\end{theorem}

\begin{proof}
Firstly, we are going to find a tilting object in $\crd$ such that each indecomposable summand belongs to the interval category
$\crd(a, \alpha(a)]$ for some $a\in \mathbb{Q}$.

By Proposition \ref{proposition2}, there is a triangle in $\crd$
\begin{equation}\label{equation7}  F_{i}[-1]\to E\langle\vec{x}_{i}+\vec{x}_{j}\rangle\oplus E\langle\vec{x}_{i}+\vec{x}_{k}\rangle\to E\langle\vec{x}_{1}+\vec{x}_{2}+\vec{x}_{3}\rangle\to F_{i},
\end{equation}
where $F_{i}$ is determined by the exact sequence
\begin{equation}\label{equation8}  0\to E\langle 2\vec{x}_{i}\rangle\to F_{i}\to \co(\vec{x}_{1}+\vec{x}_{2}+\vec{x}_{3})\to 0.
\end{equation}
Then from the following homotopy pullback commutative diagram induced by (\ref{equation7})
\begin{equation}\label{equation9}
\xymatrix{
F_{i}[-1]\ar[r]\ar@{=}[d]\ar[r]&G_{i}\ar[d]\ar[r] & E\langle\vec{x}_{j}+\vec{x}_{k}\rangle \ar[d]\ar[r]&F_{i}\ar@{=}[d] \\
F_{i}[-1]\ar[r]&E\langle\vec{x}_{i}+\vec{x}_{j}\rangle\oplus E\langle\vec{x}_{i}+\vec{x}_{k}\rangle\ar[r]& E\langle\vec{x}_{1}+\vec{x}_{2}+\vec{x}_{3}\rangle\ar[r]&F_{i},
 }
\end{equation}
we get a triangle in  $\crd$
\begin{equation}\label{equation10}  G_{i}\to \bigoplus\limits_{1\leq t\neq l\leq 3}E\langle \vec{x}_{t}+\vec{x}_{l}\rangle\xrightarrow{\beta}E\langle\vec{x}_{1}+\vec{x}_{2}+\vec{x}_{3}\rangle\to G_{i}[1].
\end{equation}
By the symmetry of the weights, we get that $\beta$ is independent of $i$, then also $G_{i}$, we denote it by $G$. Thus $G$ fits into the following triangle
\begin{equation}\label{equation11}  F_{i}[-1]\to G\to E\langle \vec{x}_{j}+\vec{x}_{k}\rangle\xrightarrow{\varphi_{i}}F_{i}.
\end{equation}

Now we claim that $G$ is determined by (\ref{equation6}). Without loss of generality, we assume that $i=2, j=1$ and $k=3$.
Notice that the nonzero morphism $\varphi_{2}: E\langle\vec{x}_{1}+\vec{x}_{3}\rangle\to F_{2}$ in $\crd$ is injective
in $\coh\X$, which fits into the following exact sequence
\begin{equation}\label{equation12}  0\to E\langle \vec{x}_{1}+\vec{x}_{3}\rangle\xrightarrow{\varphi_{2}}F_{2}\to \Coker \varphi_{2} \to 0.
\end{equation}
Since $E\langle \vec{x}_{1}+\vec{x}_{3}\rangle =E(\vec{c}-\vec{x}_{1}-\vec{x}_{3})=E(\vec{\omega}+\vec{x}_{2})$ and $\Hom(\co(2\vec{\omega}+\vec{x}_{2}), F_{2})\cong k$, we obtain the following commutative diagram
\begin{equation}\label{equation13}
\xymatrix{
0\ar[r]&\co(2\vec{\omega}+\vec{x}_{2})\ar[rd]_{\psi_{2}}\ar[r] & E\langle\vec{x}_{1}+\vec{x}_{3}\rangle \ar[d]^{\varphi_{2}}\ar[r]&\co(\vec{\omega}+\vec{x}_{2})\ar[r]&0. \\
&&F_{2}&&
 }
\end{equation}
Then by the snake lemma, we get an exact sequence
\begin{equation}\label{equation13}  0\to \co( \vec{\omega}+\vec{x}_{2})\to \Coker \psi_{2}\to \Coker \varphi_{2} \to 0.
\end{equation}
For calculation of $\Coker \psi_{2}$, we consider the exact sequence
\begin{equation}\label{equation14}  0\to \co( 2\vec{\omega}+\vec{x}_{2})\xrightarrow{\psi_{2}} F_{2}\to \Coker \psi_{2} \to 0.
\end{equation}
Notice that for any $\vec{x}> 0$, $\Hom(\co( 2\vec{\omega}+\vec{x}_{2}+\vec{x}), F_{2})=0$, we conclude that $\Coker \psi_{2}\in \vect\mathbb{X}$.
Moreover, it is easy to obtain that $[\Coker \psi_{2}]=[F_{2}]-[\co( 2\vec{\omega}+\vec{x}_{2})]=[\co( 2\vec{x}_{1}+\vec{x}_{2})]+[\co( 2\vec{x}_{3}+\vec{x}_{2})]$.
It follows that $\det(\Coker \psi_{2})=2\vec{x}_{1}+2\vec{x}_{2}+2\vec{x}_{3}$.
If $\Coker \psi_{2}$ is indecomposable, then $\mu(\Coker \psi_{2})=2$ implies that it is an Auslander bundle, that is,
$\Coker \psi_{2}=E(\vec{y})$ for some $\vec{y}\in\mathbb{L}$. Then $\det(E(\vec{y}))=\vec{\omega}+2\vec{y}=2\vec{x}_{1}+2\vec{x}_{2}+2\vec{x}_{3}$ implies that
$\vec{y}=\vec{c}$. That is $\Coker \psi_{2}=E(\vec{c})$.
But one can check that $[E(\vec{c})]\neq [\co( 2\vec{x}_{1}+\vec{x}_{2})]+[\co( 2\vec{x}_{3}+\vec{x}_{2})]$, which is a contradiction. Hence, $\Coker \psi_{2}$ is a direct sum of two line bundles.

Recall that $I(E\langle2\vec{x}_{2}\rangle)=\co( 2\vec{x}_{2})\oplus \co( \vec{x}_{2}+2\vec{x}_{3})\oplus \co( 2\vec{x}_{1}+\vec{x}_{2})\oplus \co( \vec{\omega}+\vec{c})$
and all of the direct summands vanish under the functor $\Ext^{1}(\co(\vec{x}_{1}+\vec{x}_{2}+\vec{x}_{3}), -)$ except $\co(2\vec{x}_{2})$.
Thus the exact sequence (\ref{equation8}) induces the following commutative pushout diagram
\begin{equation}\label{equation15}
\xymatrix{
0\ar[r]& E\langle2\vec{x}_{2}\rangle\ar[d]\ar[r] & F_{2}\ar[d]\ar[r]& \co(\vec{x}_{1}+\vec{x}_{2}+\vec{x}_{3})\ar@{=}[d]\ar[r]& 0\\
0\ar[r]&I(E\langle2\vec{x}_{2}\rangle)\ar[r] & E^{\prime}\ar[r]& \co(\vec{x}_{1}+\vec{x}_{2}+\vec{x}_{3})\ar[r]& 0,
 }
\end{equation}
where $E^{\prime}=(I(E\langle2\vec{x}_{2}\rangle)\backslash \co(2\vec{x}_{2}))\oplus(E\langle\vec{x}_{2}\rangle(\vec{x}_{1}+\vec{x}_{3})).$
Hence, we obtain a distinguished exact sequence
\begin{equation}\label{equation16}  0\to E\langle2\vec{x}_{2}\rangle\to F_{2}\oplus \co(2\vec{x}_{2})\to E\langle\vec{x}_{2}\rangle(\vec{x}_{1}+\vec{x}_{3}) \to 0.
\end{equation}
Thus we can easily show that the injective hull of $F_{2}$ has the expression
$I(F_{2})=\co(\vec{x}_{2}+2\vec{x}_{3})\oplus\co(2\vec{x}_{1}+\vec{x}_{2})\oplus\co(\vec{\omega}+\vec{c})\oplus \co(\vec{x}_{1}+\vec{x}_{2}+\vec{x}_{3})\oplus\co(\vec{x}_{1}+2\vec{x}_{2})\oplus\co(2\vec{x}_{2}+\vec{x}_{3}).$
By simple calculation, we find that only two direct summands of $I(F_{2})$, $\co(2\vec{x}_{1}+\vec{x}_{2})$ and $\co(\vec{x}_{2}+2\vec{x}_{3})$,
satisfy that $\det(\co(2\vec{x}_{1}+\vec{x}_{2})\oplus\co(\vec{x}_{2}+2\vec{x}_{3}))=2\vec{x}_{1}+2\vec{x}_{2}+2\vec{x}_{3}$.
It follows that $\Coker \psi_{2}=\co(2\vec{x}_{1}+\vec{x}_{2})\oplus\co(\vec{x}_{2}+2\vec{x}_{3})$.
Hence $\Coker \varphi_{2}$  fits into the following exact sequence
\begin{equation}\label{equation17}  0\to \co(\vec{\omega}+\vec{x}_{2})\to \co(2\vec{x}_{1}+\vec{x}_{2})\oplus\co(\vec{x}_{2}+2\vec{x}_{3})\to \Coker \varphi_{2} \to 0.
\end{equation}
Then the following pushout diagram
\begin{equation}\label{equation17}
\xymatrix{
0\ar[r]& \co(\vec{\omega}+\vec{x}_{2})\ar[d]^{x_{1}\cdot x_{2}}\ar[r]^{x_{2}\cdot x_{3}} & \co(2\vec{x}_{1}+\vec{x}_{2})\ar[d]\ar[r]& S_{2,1}\oplus S_{3,0}\ar@{=}[d]\ar[r]& 0\\
0\ar[r]&\co(\vec{x}_{2}+2\vec{x}_{3})\ar[r] & \co(\vec{c}+\vec{x}_{2})\oplus S_{2,1}\ar[r]&S_{2,1}\oplus S_{3,0}\ar[r]& 0
 }
\end{equation}
implies that $\Coker \varphi_{2}=\co(\vec{c}+\vec{x}_{2})\oplus S_{2,1}.$

Now we consider the following pushout diagram
\begin{equation}\label{equation18}
\xymatrix{
0\ar[r]& E\langle\vec{x}_{1}+\vec{x}_{3}\rangle\ar[d]\ar[r]& F_{2}\ar[d]\ar[r]& \Coker \varphi_{2}\ar@{=}[d]\ar[r]& 0\\
0\ar[r]&I(E\langle\vec{x}_{1}+\vec{x}_{3}\rangle)\ar[r]\ar[d] & F^{\prime}_{2}\ar[d]\ar[r]&\Coker \varphi_{2}\ar[r]& 0.\\
&E\langle\vec{x}_{1}+\vec{x}_{3}\rangle[1]\ar@{=}[r]&E\langle\vec{x}_{1}+\vec{x}_{3}\rangle[1]&&}
\end{equation}
Since each direct summand of $I(E\langle\vec{x}_{1}+\vec{x}_{3}\rangle)$ does not vanish under the functor $\Ext^{1}(\Coker \varphi_{2},-)$.
Hence, $F^{\prime}_{2}$ contains no line bundle summands. Therefore, (\ref{equation18}) induces a triangle in  $\crd$
\begin{equation}\label{equation19}  E\langle\vec{x}_{1}+\vec{x}_{3}\rangle\to F_{2}\to F^{\prime}_{2} \to E\langle\vec{x}_{1}+\vec{x}_{3}\rangle[1].
\end{equation}
Comparing with (\ref{equation11}), we obtain that $G=F^{\prime}_{2}[-1]$.
Moreover, notice that $E\langle\vec{x}_{1}+\vec{x}_{3}\rangle[1]=E\langle\vec{x}_{1}+\vec{x}_{2}+\vec{x}_{3}\rangle(\vec{x}_{2})$,
hence the fact that $\Ext^{1}(E\langle\vec{x}_{1}+\vec{x}_{2}+\vec{x}_{3}\rangle(\vec{x}_{2}), F_{2})\cong k$ implies that $G$ is determined by the following
distinguished exact sequence
\begin{equation}\label{equation20}  0\to F_{2}\to G[1] \to E\langle\vec{x}_{1}+\vec{x}_{2}+\vec{x}_{3}\rangle(\vec{x}_{2})\to 0.
\end{equation}
This finishes the proof of the claim.

Since $\beta$ in (\ref{equation10}) is the right $(\add T_{\cub}\backslash E\langle\vec{x}_{1}+\vec{x}_{2}+\vec{x}_{3}\rangle)$-approximation of $E\langle\vec{x}_{1}+\vec{x}_{2}+\vec{x}_{3}\rangle$. Hence by replacing $E\langle\vec{x}_{1}+\vec{x}_{2}+\vec{x}_{3}\rangle$ by $G$ in the tilting object $T_{\cub}$, we get that $\overline{T}=(T_{\cub}\backslash E\langle\vec{x}_{1}+\vec{x}_{2}+\vec{x}_{3}\rangle)\oplus G$ is a new tilting object in $\crd$ such that each
indecomposable direct summand belongs to $\crd(\alpha^{-1}(1), 1]$. Moreover, the endomorphism algebra $\crd(\overline{T}, \overline{T})$ has the shape
\[\xymatrix{
&E\langle\vec{x}_{1}\rangle\ar@{--}[rr]\ar[rd] & & E\langle\vec{x}_{2}+\vec{x}_{3}\rangle\\
E\ar[ru]\ar[rd]\ar[r]\ar@/^/@{--}[rr]&E\langle\vec{x}_{2}\rangle\ar[r]\ar@{--}@/_/[rr]& G\ar[r]\ar[ru]\ar[rd]&E\langle\vec{x}_{1}+\vec{x}_{3}\rangle\\
&E\langle\vec{x}_{3}\rangle\ar[ru]\ar@{--}[rr]& &E\langle\vec{x}_{1}+\vec{x}_{2}\rangle.\\
 }
\]
Then by Lemma \ref{lemma9}, $\overline{T}$ is a cluster tilting object in $\crc$ such that $\crc(\overline{T},\overline{T})$ has the shape
\[\xymatrix{
&E\langle\vec{x}_{1}\rangle\ar[rd] & & E\langle\vec{x}_{2}+\vec{x}_{3}\rangle\ar[ll]\\
E\ar[ru]\ar[rd]\ar[r]&E\langle\vec{x}_{2}\rangle\ar[r]& G\ar@/_/[ll]\ar[r]\ar[ru]\ar[rd]&E\langle\vec{x}_{1}+\vec{x}_{3}\rangle\ar@/^/[ll]\\
&E\langle\vec{x}_{3}\rangle\ar[ru]& &E\langle\vec{x}_{1}+\vec{x}_{2}\rangle.\ar[ll]\\
 }
\]
Using Keller's soft-ware, we find that by taking quiver mutations $u_{1}u_{2}u_{3}$ for the following quiver
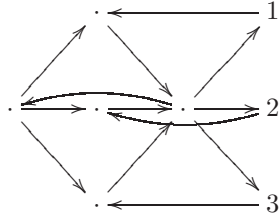
\[\xymatrix{
&\cdot\ar[rd] & & 1\ar[ll]\\
\cdot\ar[ru]\ar[rd]\ar[r]&\cdot\ar[r]& \cdot\ar@/_/[ll]\ar[r]\ar[ru]\ar[rd]&2\ar@/^/[ll]\\
&\cdot\ar[ru]& &3\ar[ll]\\
 }
\]
we obtain the quiver
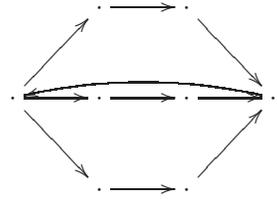
\begin{equation}\label{equation21}
\xymatrix{
&\cdot\ar[r] & \cdot\ar[rd] & \\
\cdot\ar[ru]\ar[rd]\ar[r]&\cdot\ar[r]& \cdot\ar[r]&\cdot\ar@/_/[lll]\\
&\cdot\ar[r]& \cdot\ar[ru]&\\
 }
\end{equation}

For $1\leq i\leq 3$, the quiver mutation $u_{i}$ corresponds to the following cluster mutation
\begin{equation}\label{equation22}  E\langle\vec{x}_{j}+\vec{x}_{k}\rangle[-1]\to E\langle\vec{x}_{j}+\vec{x}_{k}\rangle^{\ast}\to G \to E\langle\vec{x}_{j}+\vec{x}_{k}\rangle,
\end{equation}
where $E\langle\vec{x}_{j}+\vec{x}_{k}\rangle^{\ast}=F_{i}[-1]$ by (\ref{equation11}). Hence,
$T_{(3,3,3)}=E\oplus(\bigoplus\limits_{i=1}^{3} E\langle\vec{x}_{i}\rangle)\oplus(\bigoplus\limits_{i=1}^{3}F_{i}[-1])\oplus G$ is a cluster tilting object in $\crc$ with each indecomposable direct summand belongs to $\crd(\alpha^{-1}(1), 1].$
Thus by Lemma \ref{lemma9}, $T_{(3,3,3)}$ is a tilting object in $\crd$
with endomorphism algebra:
\begin{equation}\label{equation23}
\xymatrix{
&E\langle\vec{x}_{1}\rangle\ar[r] & F_{1}[-1]\ar[rd] & \\
E\ar[ru]\ar[rd]\ar[r]&E\langle\vec{x}_{2}\rangle\ar[r]& F_{2}[-1]\ar[r]&G,\\
&E\langle\vec{x}_{3}\rangle\ar[r]& F_{3}[-1]\ar[ru]&\\
 }\end{equation}
which is a tubular algebra of type $(3,3,3).$
\end{proof}

\vskip 10pt

\noindent {\bf Acknowledgements.} This work is partially supported by the National Natural Science Foundation of China (Grant Nos. 11571286, 11871404 and 11801473), the Natural Science Foundation of Fujian Province of China (Grant No. 2016J01031)
and the Fundamental Research Funds for the Central Universities of China (Grant Nos. 20720180002 and 20720180006).


\bibliographystyle{amsplain}

\begin{thebibliography}{20}

\bibitem{BKL} M. Barot, D. Kussin and H. Lenzing, \textit{The cluster
category of a canonical algebra.} Trans. Amer. Math. Soc.,
\textbf{362} (2010), 4313--4330.

\bibitem{BIRS} A. B. Buan, O. Iyama, I. Reiten and J. Scott,
\textit{Cluster structures for 2-Calabi-Yau categories and unipotent
groups.} Compos. Math., \textbf{145(4)} (2009), 1035--1079.

\bibitem{BMRRT} A. B. Buan, R. Marsh, M. Reineke and I. Reiten
and G. Todorov, \textit{Tilting theory and cluster combinatorics.}
Adv. Math., \textbf{204} (2006), 572--618.

\bibitem{CLR} J. Chen, Y. Lin and S. Ruan, \textit{Tilting objects in the
stable category of vector bundles on the weighted projective line of
type $(2,2,2,2;\lambda)$.}  J. Algebra.,
\textbf{397} (2014), 570--588.

\bibitem{CLLR} J. Chen, Y. Lin, P. Liu and S. Ruan, \textit{Classifications for tilting objects on a weighted projective line of type $(2,2,2,2;\lambda)$.}
arXiv:1303.1323.

\bibitem{FZ} S. Fomin and A. Zelevinsky, \textit{Cluster algebras I.
Foundations.} J. Amer. Math. Soc., \textbf{15(2)} (2002), 497--529.

\bibitem{GL} W. Geigle and H. Lenzing, \textit{A class of weighted
projective curves arising in representation theory of finite
dimensional algebras.} Singularities, representations of algebras,
and Vector bundles, Springer Lect. Notes Math., \textbf{1273}
(1987), 265--297.

\bibitem{H} D. Happel, \textit{Triangulated categories in the representation theory of finite-dimensional algebras.}
London Mathematical Society Lecture Note Series, Cambridge
University Press, Cambridge, \textbf{119} (1988).

\bibitem{IY} O. Iyama and Y. Yoshino, \textit{Mutation in triangulated categories and rigid Cohen-Macaulay modules.}
Invent. Math., \textbf{172} (2008), 117--168.

\bibitem{K} B. Keller, \textit{On triangulated orbit categories.}
Doc. Math., \textbf{10} (2005), 551--581.

\bibitem{KR} B. Keller and I. Reiten, \textit{Acyclic Calabi-Yau categories.}
Compos. Math., \textbf{144 (5)} (2008), 1332--1348.

\bibitem{KLM2} D. Kussin, H. Lenzing and H. Meltzer, \textit{Triangle
singularities, ADE-chains and weighted projective lines.} Adv. Math., \textbf{237(1)} (2013), 194--251.

\bibitem{L} H. Lenzing, \textit{Wild canonical algebras and rings of automorphic forms.} In Finite-dimensional
algebras and related topics (Ottawa, ON, 1992), volume 424 of NATO Adv. Sci. Inst. Ser.
C Math. Phys. Sci., 191--212. Kluwer Acad. Publ., Dordrecht, 1994.

\bibitem{LM} H. Lenzing and H. Meltzer, \textit{Sheaves on a weighted
projective line of genus one, and representations of a tubular
algebra.} In Representations of algebras, Sixth International
Conference, Ottawa 1992. CMS Conf. Proc., \textbf{14} (1993),
313--337.

\bibitem{LR} H. Lenzing and S. Ruan, \textit{On vector bundles of rank two on a weighted
projective line.} In Preparation.

\bibitem{O} D. Orlov, \textit{Derived categories of coherent sheaves and
triangulated categories of singularities.} Algebra, Arithmetic, and Geometry,
Progress in Mathematics, \textbf{270} (2009), 503-531.

\end{thebibliography}

\end{document}